\documentclass[12pt]{amsart}
\usepackage{amsmath} 
\usepackage{amsthm}
\usepackage{geometry}                
\usepackage[parfill]{parskip}    
\usepackage{graphicx}
\usepackage[noend]{algorithm2e}
\usepackage{amssymb}
\usepackage{epstopdf}
\usepackage{verbatim}
\usepackage{hyperref}
\hypersetup{colorlinks,%
citecolor=black,%
filecolor=black,%
linkcolor=black,%
urlcolor=black%
}
\newtheorem{thm}{Theorem}[section]

\theoremstyle{definition}

\theoremstyle{remark}
\newtheorem*{rem}{Remark}
\newcommand{\f}[1]{\mathbb{F}_{#1}}
\newcommand{\tfr}[1]{{\widehat{#1}}}

\DeclareMathOperator{\Tr}{Tr}
\DeclareMathOperator{\wt}{wt}

\begin{document}
\baselineskip=17pt

\title[Binary Kloosterman sums]{Binary Kloosterman sums using Stickelberger's theorem and the Gross-Koblitz formula}

\author[F. G\"olo\u{g}lu]{Faruk G\"olo\u{g}lu}

\author[G. McGuire]{Gary McGuire}
\author[R. Moloney]{Richard Moloney}

\address{School of Mathematical Sciences\\
	University College Dublin\\
		Ireland\\}
\email{farukgologlu@gmail.com, gary.mcguire@ucd.ie, richard.moloney@ucd.ie}

\date{}


\begin{abstract}
Let $\mathcal{K}(a)$ denote the Kloosterman sum on the finite field of order $2^n$. 
We give a simple characterization of  $\mathcal{K}(a)$ modulo $16$,
in terms of the trace of $a$ and one other function.
We also give a characterization of $\mathcal{K}(a)$ modulo $64$ in terms of 
a different function, which we call the lifted trace.
Our proofs use Fourier analysis, Stickelberger's theorem and the Gross-Koblitz formula.
\end{abstract}




\maketitle

\section{Introduction}

Let $\mathcal{K}_{p^n}(a)$ denote the $p$-ary Kloosterman sum defined by
\[
\mathcal{K}_{p^n}(a) := \sum_{x \in \f{p^n}} \zeta^{\Tr (x^{p^n-2}+ax)},
\]
for any $a \in \f{p^n}$, where $\zeta$ is a primitive $p$-th root of unity 
and $\Tr$ denotes the absolute trace map $\Tr : \f{p^n} \to \f{p}$ defined as usual as
\[
\Tr (c) := c + c^p + c^{p^2} + \cdots + c^{p^{n-1}}.
\]

Finding explicit zeros (explicit $a$'s with $\mathcal{K}_{p^n}(a)=0$) 
of Kloosterman sums is considered difficult. Recent research on Kloosterman sums is generally concentrated on proving divisibility results and characterisation of Kloosterman sums modulo some integer (see \cite{MMeven, LisonekSeta, Div3, CHZ, Lisonek2009}). 

It is easy to see that binary Kloosterman sums are divisible by $4 = 2^2$, i.e.,
for all $a \in \f{2^n}$,
\begin{equation}\label{folklore}
\mathcal{K}_{2^n}(a) \equiv 0 \pmod{4}.
\end{equation}

They also satisfy (see \cite{LachWolf})
\[
-2^{n/2+1} \le \mathcal{K}_{2^n}(a) \le 2^{n/2+1},
\]
and take every value which is congruent to $0$ modulo $4$ in that range.

Helleseth and Zinoviev proved the following result which improved \eqref{folklore} 
one level higher, i.e., modulo $2^3$,  in the sense
of describing the $a$ for which $\mathcal{K}_{2^n}(a)$ is 0 or 4 modulo 8.

\bigskip

\begin{thm}\cite{HZ}\label{mod8}
For $a \in \f{2^n}$,
\[
\mathcal{K}_{2^n}(a) \equiv \left\{ \begin{array}{ll} 
0 \ (\mathrm{mod} \ 8) & \textrm{if } \Tr(a) = 0,\\
4 \ (\mathrm{mod} \ 8) & \textrm{if } \Tr(a) = 1. 
\end{array}
\right. 
\]
\end{thm}

This paper will improve Theorem \ref{mod8} to higher levels, i.e., modulo $2^4$,
in the sense of 
describing the residue class of $\mathcal{K}_{2^n}(a)$ modulo $2^4$ in terms of $a$.
 We will define the quadratic sum  
\[
Q(a) := \sum_{0 \le i < j < n} a^{2^i+2^j}.
\] 
While the trace map $\Tr(a)$ is the sum of all linear powers of $a$, 
the sum $Q(a)$ is the sum of all quadratic powers of $a$. 
Using Stickelberger's theorem we will improve 
the Helleseth-Zinoviev result one level further to the modulus $2^4$.
We will prove the following theorem.

\bigskip

\begin{thm}\label{mainthm}
For $a \in \f{2^n}$, 
\begin{equation*}
\mathcal{K}(a)\equiv \left\{
	\begin{array}{rccl}
			0 \pmod{16}& \text{ if } &\Tr(a) = 0 \text{ and } &Q(a) = 0,\\
			4\pmod{16}& \text{ if } &\Tr(a) = 1 \text{ and } &Q(a) = 1,\\
			8\pmod{16}& \text{ if } &\Tr(a) = 0 \text{ and } &Q(a) = 1,\\
			12\pmod{16}& \text{ if } &\Tr(a) = 1 \text{ and } &Q(a) = 0.
		\end{array} \right.
		\end{equation*}
\end{thm}

We mention a recent result due to Lisonek \cite{LisonekSeta}
that gives a description of the 
elements $a \in \f{2^n}$ for which $\mathcal{K}(a) \equiv 0 \pmod{16}$:
\bigskip

\begin{thm}
Let $n \ge 4$. For any $a \in \f{2^n}, \mathcal{K}(a)$ is divisible by $16$ if and only if $\Tr(a) = 0$ and $\Tr(y) = 0$ where $y^2 + ay + a^3 = 0$.  
\end{thm}

In Sections \ref{sect::Stickelberger} and \ref{sect::Fourier}, we will introduce the techniques we use. In Section \ref{sect::Mod8} we give an alternative proof of  Theorem \ref{mod8} using our techniques.  
We  prove Theorem \ref{mainthm} in Section \ref{sect::Mod16}. 
In Section \ref{sect::Mod48} we will combine Theorem \ref{mainthm} with the result concerning Kloosterman sums modulo $3$ to achieve the complete characterisation modulo $48$. 
Finally, in Section \ref{sect::Mod64} we will employ the Gross-Koblitz formula to characterize the values of Kloosterman sums modulo $64$ in terms of 
the {\em lifted trace} that we will introduce in Section \ref{sect::Mod16}. 

We give a few remarks about the ternary case. 
It is easy to see that ternary Kloosterman sums are divisible by $3$, i.e.,
for all $a \in \f{3^n}$,
\begin{equation}\label{folklore2}
\mathcal{K}_{3^n}(a) \equiv 0 \pmod{3}.
\end{equation}

Ternary Kloosterman sums satisfy (see Katz and Livn\'e  \cite{Katz})
\[
-2\sqrt{3^n} < \mathcal{K}_{3^n}(a) < 2\sqrt{3^n}
\]
and take every value which is congruent to $0$ modulo $3$ in that range.

In a recent paper, we used Stickelberger's theorem to prove the following 
result on ternary Kloosterman sums, which improved  \eqref{folklore2} 
one level higher.

\bigskip

\begin{thm}\cite{SetaFGR}
For $a \in \f{3^n}$,
\[
\mathcal{K}_{3^n}(a) \equiv \left\{ \begin{array}{ll} 
0 \ (\mathrm{mod} \ 9) & \textrm{if } \Tr(a) = 0,\\
3 \ (\mathrm{mod} \ 9) & \textrm{if } \Tr(a) = 1,\\ 
6 \ (\mathrm{mod} \ 9) & \textrm{if } \Tr(a) = 2. 
\end{array}
\right. 
\]
\end{thm}

\section{Stickelberger's theorem}\label{sect::Stickelberger}

Let $p$ be a prime (in Section \ref{sect::Mod8} we set $p=2$) and let $q = p^n$.
We consider multiplicative characters taking their values in an algebraic extension of 
the $p$-adic numbers $\mathbb{Q}_p$. 
Let $\xi$ be a primitive $(q-1)^{\text{th}}$ root of unity in a 
fixed algebraic closure of  $\mathbb{Q}_p$.  
The group of multiplicative characters of $\mathbb{F}_q$ (denoted $
\widehat{\mathbb{F}_q^{\times}}$) is cyclic of order $q-1$. 
The group $\widehat{\mathbb{F}_q^
{\times}}$ is generated by the Teichm\"uller character $\omega : \mathbb{F}_q^\times \to 
\mathbb{Q}_p(\xi)$, which, for a fixed generator $t$ of $\mathbb{F}_q^{\times}$, is defined 
by 
\[
\omega (t^j) = \xi^j.
\] 
We extend $\omega$ to $\mathbb{F}_q$ by setting $\omega (0)$ to be 0.

Let $\zeta$ be a  primitive $p$-th root of unity in the 
fixed algebraic closure of  $\mathbb{Q}_p$.  
Let $\mu$ be the canonical additive character of $\mathbb{F}_q$,
\[
\mu(x) = \zeta^{\Tr (x)}.
\] 

The Gauss sum (see \cite{LN,Wash}) of a character $\chi \in \widehat{\mathbb{F}_q^
{\times}}$  is defined as \[\tau(\chi) = -\sum_{x \in \mathbb{F}_q}\chi(x)\mu(x)\,.\]

For any positive integer $j$, let $\wt_p(j)$ denote the $p$-weight of $j$, i.e.,
\[ \wt_p(j)=\sum_i j_i
\] 
where $\sum_i j_ip^i$ is the $p$-ary
expansion of $j$. Just for shorthand notation we define \[g(j):=\tau(\omega^{-j})=\tau(\bar{\omega}^j)\,.\]

Let $\pi$ be the unique ($p-1$)th root of $-p$ in $\mathbb{Q}_p(\xi,\zeta)$ satisfying \[\pi \equiv \zeta - 1 \pmod{\pi^{2}}\,.\]
Wan \cite{Wan}
noted that the following improved version of Stickelberger's theorem is a direct consequence of the Gross-Koblitz formula \cite{GK, Rob}.
\bigskip

\begin{thm}\cite{Wan}\label{Stick+}
Let $1 \le j < q-1$ and let $j =j_0+j_1p+\dots + j_{n-1}p^{n-1}$. Then 
\[g(j) \equiv \frac{\pi^{\wt_p(j)}}{j_{0}!\cdots j_{n-1}!} \pmod{\pi^{\wt_p(j)+p-1}}\,.\]
\end{thm}

Stickelberger's theorem, as usually stated,  is the same congruence modulo $\pi^{\wt_p(j)+1}$.
Note that when $p=2$, which is the case in this paper, Theorem \ref{Stick+}
is the same as this original Stickelberger theorem.

We have (see \cite{GK}) that
$(\pi)$ is the unique prime ideal of $\mathbb{Q}_{p}(\zeta, \xi)$ lying above $p$. 
Since $\mathbb{Q}_{p}(\zeta, \xi)$ is an unramified extension of $\mathbb{Q}_{p}(\zeta)$,
a totally ramified (degree $p-1$) extension of $\mathbb{Q}_{p}$, it follows that $(\pi)^{p-1} = (p)$
and $\nu_p(\pi) = \frac{1}{p-1}$.
Here $\nu_p$ denotes the $p$-adic valuation.

Therefore Theorem \ref{Stick+} implies that
$\nu_{\pi}(g(j)) = \wt_p(j)$, and because 
$\nu_{p}(g(j)) =\nu_{\pi}(g(j)) \cdot \nu_p (\pi)$ we get 
\begin{equation}\label{val}
\nu_p(g(j)) = \frac{\wt_p(j)}{p-1}\,.\end{equation}

In this paper we have $p=2$.  In that case, $\pi=-2$ and  
equation \eqref{val} becomes
 \begin{equation}\label{val2}
\nu_2(g(j)) = \wt_2(j)\,.\end{equation}

\section{Fourier Analysis}\label{sect::Fourier}

The Fourier transform of a function $f:\mathbb{F}_q \to \mathbb{C}$ at $a \in \mathbb{F}_q$ 
is defined to be
\[\tfr{f}(a) = \sum_{x \in \mathbb{F}_q} f(x)\mu(ax)\,.\]
The complex number $\tfr{f}(a) $ is called the Fourier coefficient of $f$ at $a$.

Consider monomial functions defined by $f(x) = \mu(x^d)$. 
When $d=-1$ we have $\tfr{f}(a) = \mathcal{K}_{p^n}(a)$.
By a similar 
Fourier analysis argument to that in Katz \cite{KatzGaussSums} or
Langevin-Leander \cite{LL}, for any $d$ we have
\[
\tfr{f}(a) =\frac{q}{q-1}+ \frac{1}{q-1} \sum_{j=1}^{q-2}\tau(\bar{\omega}^j)\ \tau
(\omega^{jd})\ \bar{\omega}^{jd}(a) 
\]
and hence
\[
\tfr{f}(a) \equiv - \sum_{j=1}^{q-2}\tau(\bar{\omega}^j)\ \tau
(\omega^{jd})\ \bar{\omega}^{jd}(a) \pmod{q}\,.
\]
We will use this to obtain congruence information about  Kloosterman sums.
Putting $d=-1=p^n-2$, the previous congruence becomes

\begin{equation}\label{firstKL}
\mathcal{K}(a) \equiv - \sum_{j=1}^{q-2}(g(j))^2\ \omega^{j}(a) \pmod{q}.
\end{equation}
Equation (\ref{val2}) gives the $2$-adic valuation of the Gauss sums $g(j)$,
and the $2$-adic valuation of each term in  equation (\ref{firstKL}) follows.
Our proofs will consider (\ref{firstKL}) at various levels, i.e., modulo $2^3$, $2^4$ and $2^6$.

\section{Binary Kloosterman sums modulo 8}\label{sect::Mod8}

Let $q = 2^n$ for some integer $n\geq 2$.

To warm up we shall give a new proof of the following result due to Helleseth 
and Zinoviev \cite{HZ}.
This is equivalent to Theorem \ref{mod8}.

\bigskip

\begin{thm}\label{mod8lem}
For $a\in \mathbb{F}_q$, $\mathcal{K}(a) \equiv 0 \pmod{8}$ if and only if  $\Tr (a) = 0$.
\end{thm}

\begin{proof}
If $f(x) = \mu(x^d)$ let
\[M_{d} = \min_{j\in\{1,2,\dots2^n-2\}}[\wt_2(j)+\wt_2(-jd)]\,,\]
and let
\[J_{d} = \left\{j\in\{1,2,\dots2^n-2\} : \wt_2(j)+\wt_2(-jd) = M_{d}\right\}\,.\]

Lemma 1 of \cite{LLMZ} states that if $f(x) = \mu(x^d)$, then 
\begin{equation}\label{lemma1ofLLMZ}
2^{M_d+1}\mid \tfr{f}(a) 
\iff \sum_{j \in J_d}a^{-jd} = 0.
\end{equation} 
 
 Let $d = -1$. Then $\tfr{f}(a)$ is the Kloosterman sum $\mathcal{K}(a)$
 on $\f{q}$, $M_{-1} = 2$, and \[J_{-1} = \left\{j\in\{1,2,\dots2^n-2\} : \wt_2(j) = 1\right\}.\] 
It follows that
\[\sum_{j \in J_{-1}}a^{j} = \Tr (a)\,,\]
and \eqref{lemma1ofLLMZ} implies that $8$ divides $\mathcal{K}(a)$ if and only if $\Tr (a)=0$.
\end{proof}


\section{Binary Kloosterman sums modulo 16}\label{sect::Mod16}

Again $q=2^n$.
For $i= 1,2,\dots$, let \[W_{i} = \left\{j\in\{1,2,\dots2^n-2\} : \wt_2(j) = i\right\}.\]
Then we may write
\[\Tr(a) = \sum_{j \in W_{1}}a^j.
\] 
Recall that $\omega : \mathbb{F}_q \longrightarrow \mathbb{Q}_2 (\xi)$ is the
Teichm\"uller character.

We define the \emph{lifted trace} 
$\tfr{\Tr}:\mathbb{F}_{q} \to \mathbb{Q}_2(\xi)$ by 
\[\tfr{\Tr}(a) = \sum_{j \in W_{1}}\omega(a^j)
\] 
and note that $\tfr{\Tr}(a) \equiv \Tr(a) \pmod{2}$. 

We define the \emph{quadratic trace} $Q: \mathbb{F}_q \longrightarrow \mathbb{F}_2$ by
\[Q(a) = \sum_{j \in W_{2}} a^{j}\] 
and define the  \emph{lifted quadratic trace}
$\tfr{Q}:\mathbb{F}_{q} \to \mathbb{Q}_2(\xi)$ by 
\[\tfr{Q}(a) = \sum_{j \in W_{2}}\omega(a^j).\] 
Then $\tfr{Q}(a) \equiv Q(a) \pmod{2}$.

Next we prove our theorem on $\mathcal{K}(a)$ mod 16.
\bigskip

\begin{thm}\label{thm::mod16} Let $q = 2^n$. For $a \in \f{q}$, 
\begin{equation*}
\mathcal{K}(a)\equiv \left\{
	\begin{array}{rccl}
			0 \pmod{16}& \text{ if } &\Tr(a) = 0 \text{ and } &Q(a) = 0,\\
			4\pmod{16}& \text{ if } &\Tr(a) = 1 \text{ and } &Q(a) = 1,\\
			8\pmod{16}& \text{ if } &\Tr(a) = 0 \text{ and } &Q(a) = 1,\\
			12\pmod{16}& \text{ if } &\Tr(a) = 1 \text{ and } &Q(a) = 0.
		\end{array} \right.
		\end{equation*}
\end{thm}
\begin{proof}
Let $q=2^n$ and let $a\in \mathbb{F}_q$.
As in the proof of Lemma \ref{mod8lem}, $\mathcal{K}(a) = \tfr{f}(a)$, where $f(x) = \mu(x^{-1})$. 
Stickelberger's theorem implies $g(j) \equiv 2^{\wt_2(j)}\pmod{2^{\wt_2(j)
+1}}$, so squaring gives 
\[(g(j))^2 \equiv 2^{2\wt_2(j)}\pmod{2^{2\wt_2(j)+2}}\,.\]
It follows that $g(j)^2 \equiv 4 \pmod{16}$ for $j$ of weight 1,
and $g(j)^2 \equiv 0 \pmod{16}$ for $j$ of weight at least 2.
Thus congruence \eqref{firstKL} modulo 16 gives
\[\mathcal{K}(a) \equiv - \sum_{j \in W_1}g(j)^2\omega^j(a) \pmod{16}\]
or in other words
\[\mathcal{K}(a) \equiv -4\ \tfr{\Tr}(a) \pmod{16}.\] 
It remains to  determine $\tfr{\Tr}(a) \bmod{4}$.

This can be done in terms of the $\f{q}$-sums $\Tr(a)$ and $Q(a)$ by noting that\begin{align*}
\tfr{\Tr}(a)^2 &= \sum_{j \in W_1}\sum_{k \in W_1}\omega(a^j)\omega(a^k)\\
&=\sum_{j, k \in W_1}\omega(a^{j+k})\\
&=2\sum_{i \in W_2}\omega(a^i) + \sum_{j \in W_1}\omega(a^j)\\
&= 2\tfr{Q}(a)+\tfr{\Tr}(a).
\end{align*}
However
\[\tfr{\Tr}(a)^2 \equiv 0 \pmod{4}\iff \tfr{\Tr}(a) \equiv 0 \pmod{2}
\iff \Tr (a) = 0\]
and 
\[\tfr{\Tr}(a)^2 \equiv 1 \pmod{4}\iff \tfr{\Tr}(a) \equiv 1 \pmod{2}
\iff \Tr (a) = 1.\]
Recalling that $\tfr{Q}(a) \equiv Q(a) \pmod{2}$, and observing that we only require $\tfr{Q}(a)$ mod 2,
we get
\begin{equation*}
\tfr{\Tr}(a)\equiv \left\{
	\begin{array}{rccl}
			0\pmod{4}& \text{ if } &\Tr(a) = 0 \text{ and } &Q(a) = 0,\\
			1\pmod{4}& \text{ if } &\Tr(a) = 1 \text{ and } &Q(a) = 0,\\
			2\pmod{4}& \text{ if } &\Tr(a) = 0 \text{ and } &Q(a) = 1,\\
			3\pmod{4}& \text{ if } &\Tr(a) = 1 \text{ and } &Q(a) = 1,
		\end{array} \right.
		\end{equation*} which proves the result.\qedhere \end{proof}

\section{Binary Kloosterman sums modulo $48$}\label{sect::Mod48}
We combine the results above with the result on the divisibility modulo 3 of binary Kloosterman 
sums from \cite{CHZ}, \cite{Div3}, \cite{MMAA} and \cite{MMeven} to fully characterise the 
congruence modulo 48 of binary Kloosterman sums.

\subsection{Case $n$ odd}
\begin{thm}
Let $q = 2^n$ and let $a \in \f{q}^{\times}$ where $n$ is odd and $n \ge 5$.
\begin{enumerate}
\item If $\Tr (a^{1/3}) = 0$ then
\begin{equation*}
\mathcal{K}(a)\equiv \left\{
	\begin{array}{rccl}
			4\pmod{48}& \text{ if } &\Tr(a) = 1 \text{ and } &Q(a) = 1,\\
			16\pmod{48}& \text{ if } &\Tr(a) = 0 \text{ and } &Q(a) = 0,\\
			28\pmod{48}& \text{ if } &\Tr(a) = 1 \text{ and } &Q(a) = 0,\\
			40\pmod{48}& \text{ if } &\Tr(a) = 0 \text{ and } &Q(a) = 1, 
		\end{array}\right.
\end{equation*}
\item If $\Tr (a^{1/3}) = 1$, let $\beta$ be the unique element satisfying $\Tr(\beta) = 0$, 
$a^{1/3} = \beta^4+\beta+1$. Then 
\begin{equation*}
\mathcal{K}(a)\equiv \left\{
	\begin{array}{rcccl}
			0\pmod{48}& \text{ if } &\Tr(a) = 0, &Q(a) = 0, &n+\Tr (\beta^3) \equiv 5,7 \ (8),\\
			8\pmod{48}& \text{ if } &\Tr(a) = 0, &Q(a) = 1,  &n+{\Tr} (\beta^3) \equiv 1,3\ (8),\\
			12\pmod{48}& \text{ if } &\Tr(a) = 1, &Q(a) = 0,   &n+ \Tr (\beta^3) \equiv 5,7\ (8),\\
			20\pmod{48}& \text{ if } &\Tr(a) = 1, &Q(a) = 1,  &n+ \Tr (\beta^3) \equiv 1,3\ (8),\\
			24\pmod{48}& \text{ if } &\Tr(a) = 0, &Q(a) = 1,   &n+\Tr (\beta^3) \equiv 5,7\ (8),\\
			32\pmod{48}& \text{ if } &\Tr(a) = 0, &Q(a) = 0,  &n+\Tr (\beta^3) \equiv 1,3 \ (8),\\
			36\pmod{48}& \text{ if } &\Tr(a) = 1, &Q(a) = 1, &n+\Tr (\beta^3) \equiv 5,7\ (8),\\
			44\pmod{48}& \text{ if } &\Tr(a) = 1, &Q(a) = 0,  &n+\Tr (\beta^3) \equiv 1,3 \ (8).
		\end{array} \right.
\end{equation*}Note that  we consider $\Tr(\beta^3)$ to be an integer in the final congruences.
 
\end{enumerate}
\end{thm}
\begin{proof} Follows from Theorem \ref{thm::mod16} above, and theorem 3 of \cite{CHZ}, which implies that
$\mathcal{K}(a) \equiv 1 \pmod{3} \iff \Tr(a^{1/3}) = 0$ and 
otherwise,  $\mathcal{K}(a) \equiv 0 \pmod{3}$ if and only if   either $\Tr(\beta^3)=0$
and $n \equiv 5 \text{ or } 7 \pmod{8}$,
or $ \Tr(\beta^3)=1 \text{ and } n \equiv 1 \text{ or } 3 \pmod{8}$.
\end{proof}

\subsection{Case $n$ even}
By a similar argument (with a few more cases) we can combine Theorem \ref{thm::mod16} above with Theorem 11 of \cite
{MMeven} to classify the congruence modulo 48 of the Kloosterman sum on $\mathbb{F}_
{2^n}$ where $n$ is even.  We omit the details.

\section{Binary Kloosterman sums modulo 64}\label{sect::Mod64}

So far in this paper we have used 
the lifted trace modulo 2 (the usual finite field trace) and the lifted quadratic trace modulo 2 
to characterize the Kloosterman sums modulo 16.
Further information can be obtained using the lifted traces modulo higher powers of 2.
We will now show how the values taken by the lifted trace modulo 16 determine the congruence modulo 64 of binary Kloosterman sums, using the Gross-Koblitz formula.

The first part of this section, down to theorem \ref{thm::mod64}, is a restatement of Section 8 of \cite{LLMZ} (with a correction when $q=4$).

For a field $\mathbb{F}_{q}=\mathbb{F}_{2^n}$, and $j$ a residue modulo $q-1$, 
the {\bf Gross-Koblitz formula} \cite{Rob} states that 
\begin{equation}\label{GKformula}
\tau (\bar{\omega}^j) = (-2)^{\wt_2(j)}\prod_{i=0}^{n-1}\Gamma_{2}\left(\left\langle \frac{2^{i}j}{q-1} \right\rangle\right)
\end{equation}
where $\langle x \rangle$ is the fractional part of $x$, and $\Gamma_2$ is the 2-adic Gamma function.

The $p$-adic Gamma function $\Gamma_{p}$ 
is defined over $\mathbb{N}$ by \[\Gamma_{p}(k) = (-1)^k\prod_{\substack{t<k\\(t,p) = 1}}t \,.\]
By the generalised Wilson's theorem, $\Gamma_{p}(p^k) \equiv 1 \pmod{p^k}$, unless $p^k = 4$, in which case $\Gamma_{2}(4) \equiv -1 \pmod{4}$.

Suppose $x \equiv y \pmod{2^k}$. 
Observe that $(-1)^{x+2^k} = (-1)^{x}$, and that the product \[\prod_{\substack{x\le t <x+2^k\\ (t,2) = 1}}t\ \bmod{2^k}\] consists of $2^{k-1}$ 
distinct elements, and hence is congruent to $\Gamma_2(2^k)$.
It follows that $\Gamma_2(x) \equiv \Gamma_2(y) \pmod{2^k}$ unless $k = 2$, in which case $\Gamma_2(x) \equiv -\Gamma_2(y) \pmod{4}$. 
\bigskip

\begin{thm}\label{thm::mod64} Let $n\ge 6$ and let $q = 2^n$. For $a \in \f{q}$,
\begin{equation*}
\mathcal{K}(a)\equiv \left\{
	\begin{array}{rcccl}
			0\pmod{64}&\text{ if } &\tfr{\Tr}(a) \equiv &0 &\pmod{16}\\
			4\pmod{64}&\text{ if } &\tfr{\Tr}(a) \equiv &11 &\pmod{16}\\
			8\pmod{64}&\text{ if } &\tfr{\Tr}(a) \equiv &10 &\pmod{16}\\
			12\pmod{64}&\text{ if } &\tfr{\Tr}(a) \equiv &13 &\pmod{16}\\
			16\pmod{64}&\text{ if } &\tfr{\Tr}(a) \equiv &4 &\pmod{16}\\
			20\pmod{64}&\text{ if } &\tfr{\Tr}(a) \equiv &15 &\pmod{16}\\
			24\pmod{64}&\text{ if } &\tfr{\Tr}(a) \equiv &14 &\pmod{16}\\
			28\pmod{64}&\text{ if } &\tfr{\Tr}(a) \equiv &1 &\pmod{16}\\
			32\pmod{64}&\text{ if } &\tfr{\Tr}(a) \equiv &8 &\pmod{16}\\
			36\pmod{64}&\text{ if } &\tfr{\Tr}(a) \equiv &3 &\pmod{16}\\
			40\pmod{64}&\text{ if } &\tfr{\Tr}(a) \equiv &2 &\pmod{16}\\
			44\pmod{64}&\text{ if } &\tfr{\Tr}(a) \equiv &5 &\pmod{16}\\
			48\pmod{64}&\text{ if } &\tfr{\Tr}(a) \equiv &12 &\pmod{16}\\
			52\pmod{64}&\text{ if } &\tfr{\Tr}(a) \equiv &7 &\pmod{16}\\
			56\pmod{64}&\text{ if } &\tfr{\Tr}(a) \equiv &6 &\pmod{16}\\
			60\pmod{64}&\text{ if } &\tfr{\Tr}(a) \equiv &9 &\pmod{16}.
		\end{array} \right.
\end{equation*}
\end{thm}

\begin{proof}  By the statements above,
the following congruences hold for residues mod $8$:
\begin{align*}
&\Gamma_2(0) \equiv 1 \pmod{8}\\
&\Gamma_2(1) \equiv 7 \pmod{8}\\
&\Gamma_2(2) \equiv 1 \pmod{8}\\
&\Gamma_2(3) \equiv 7 \pmod{8}\\
&\Gamma_2(4) \equiv 3 \pmod{8}\\
&\Gamma_2(5) \equiv 5 \pmod{8}\\
&\Gamma_2(6) \equiv 7 \pmod{8}\\
&\Gamma_2(7) \equiv 1 \pmod{8}\,.
\end{align*}

By Lemma 6.5 of \cite{Wash}, $g(2^i) = g(1)$, and a simple calculation gives 
\[\prod_{i=0}^{n-1}\Gamma_{2}\left(\left\langle \frac{2^{i}j}{q-1} \right\rangle\right) \equiv 5 \pmod{8}\,.\]
Thus \[g(1) \equiv 6 \pmod{16}\]which implies
\[g(1)^2 \equiv 36 \pmod{64}\,.\]

Now using Stickelberger's theorem (Theorem \ref{Stick+}), we see that for $j$ of weight 2, \[g(j) \equiv 4 \pmod{8}\] and thus\[g(j)^2 \equiv 16 \pmod{64}\,.\]

Taking this into account, reading congruence \eqref{firstKL} modulo 64 gives
\[
\mathcal{K}(a) \equiv -36\ \tfr{\Tr}(a) -16\ \tfr{Q}(a) \pmod{64}\,.
\]

As we have noted, \[2\tfr{Q}(a) = \tfr{\Tr}(a)^2- \tfr{\Tr}(a)\,,\] so the value of $\tfr{\Tr}(a) \bmod{16}$ determines $\tfr{Q}(a) \bmod{8}$, and so determines $16\ \tfr{Q}(a) \pmod{64}$.
Thus $\tfr{\Tr}(a) \bmod{16}$ completely determines $\mathcal{K}(a)$ mod 64.
The possibilities are enumerated in the statement.
\end{proof}

\begin{rem}Just as we did in Section \ref{sect::Mod48}, this theorem can be combined with the results on binary Kloosterman sums modulo 3 to yield a theorem characterizing binary Kloosterman sums modulo 192.  We omit the details.\end{rem}

\subsection*{Acknowledgements}

This research was supported by the Claude Shannon Institute, Science Foundation Ireland Grant 06/MI/006 and, in the case of the third author, the Irish Research Council for Science, Engineering and Technology.

\bibliographystyle{siam}
\bibliography{main}
\end{document}